\documentclass[12pt]{article}
\usepackage[english]{babel}
\usepackage[numbers,sort&compress]{natbib}
\usepackage{t1enc}
\usepackage{amsthm,amsmath,amssymb}
\usepackage{mathrsfs}
\usepackage[mathscr]{eucal}
\usepackage[utf8]{inputenc}
\usepackage{latexsym}
\usepackage{graphicx}
\newtheorem{theorem}{Theorem}
\newtheorem{proposition}{Proposition}
\newtheorem{lemma}{Lemma}
\newtheorem{problem}{Problem}
\def \T{\textup{T}}
\def \diag{\textup{diag}}
\begin{document}
	\title{On a generalization of the Johnson-Newman theorem to multiple rank-one perturbations}
\author{Wei Wang\thanks{Corresponding author. Email address: wangwei.math@gmail.com}\quad \quad Siqi Wu\\
{\footnotesize  School of Mathematics, Physics and Finance, Anhui Polytechnic University, Wuhu 241000, P. R. China
}
}
\date{}

\maketitle
\begin{abstract}
Wang and Zhao (Adv. Appl. Math. 173 (2026) 102994) generalized the classic Johnson-Newman theorem  on simultaneous similarity of symmetric matrices  from a single rank-one perturbation to multiple rank-one perturbations. However, their result applies only to specific rank-one perturbations, and the given condition is quite involved as it relies on multivariate polynomials. We provide a simple proof of their result, leading to an improved version with a simplified condition that holds for arbitrary rank-one perturbations.
	
\end{abstract}
	

\section{Introduction}

Let  $\mathcal{A}:=(A_0,\ldots,A_{m})$ and $\mathcal{B}:=(B_0,\ldots, B_{m})$ be two $(m+1)$-tuples of real symmetric matrices. We call  $\mathcal{A}$ and $\mathcal{B}$ \emph{entrywise  similar} if for  each  $i\in \{0,1,\ldots,m\}$, the corresponding matrices $A_i$ and $B_i$ are similar, i.e.,  there exists an orthogonal matrix $Q_i$ such that $Q_i^\T A_iQ_i=B_i$.  We call $\mathcal{A}$ and $\mathcal{B}$ \emph{simultaneously similar} if there exists a common orthogonal matrix $Q$ such that $Q^\T A_i Q=B_i$ for all $i\in\{0,1,\ldots,m\}$.  In a classic paper concerning cospectral graphs, Johnson and Newman \cite{johnson1980JCTB} proved the following result on simultaneous similarity for the special case that $\mathcal{A}=(A,A+ee^\T)$ and $\mathcal{B}=(B,B+ee^\T)$, where $e$ is the all-ones vector.  
\begin{theorem}[\cite{johnson1980JCTB}]\label{jper}
Let $A$ and $B$ be two $n\times n$ real symmetric matrices. Then  $(A,A+ee^\T)$ and $(B,B+ee^\T)$ are entrywise similar if and only if they are simultaneously similar via an orthogonal matrix $Q$ such that $Q^\T A Q=B$ and $Q^\T e=e$.
\end{theorem}
Assuming $A$ and $B$ are adjacency matrices of two graphs $G$ and $H$, Wang \cite{wang2006EUJC} observed that if $G$ is controllable then  the corresponding orthogonal matrix $Q$ is unique and rational. Based on this elementary but important observation, Wang and his collaborators have successfully developed the theory of generalized spectral characterizations of graphs, see e.g. \cite{wang2006EUJC,wang2017JCTB,guo2025AAM,qiu2023}

In a recent paper concerning the spectral characterizations of regular graphs, Qiu et al.~\cite{qiu2023} found an extension of Theorem \ref{jper}: the all-ones matrix $ee^\T$ in Theorem \ref{jper} can be replaced by any  symmetric  matrices  of the form $\xi \xi^\T$, where $\xi\in \mathbb{R}^n\setminus \{0\}$. 
\begin{theorem}[\cite{qiu2023}]\label{rk1per}
	Let $A,B$ be two $n\times n$ real symmetric matrices and let $\alpha,\beta$ be two nonzero vectors in $\mathbb{R}^n$. Then  $(A,A+\alpha\alpha^\T)$ and $(B,B+\beta\beta^\T)$ are entrywise similar if and only if  they are  simultaneously similar via an orthogonal matrix $Q$ such that $Q^\T \alpha=\beta$.
\end{theorem} 

Note that both Theorems \ref{jper} and \ref{rk1per} only consider simultaneous similarity for $2$-tuples, which play  key roles in studying the problem of generalized spectral characterizations of  graphs. Motivated by a  refinement of generalized spectral characterization of graphs,  the following natural problem was suggested in a recent paper of Wang and Zhao \cite{wang2026AAM}. 
\begin{problem}[\cite{wang2026AAM}]\label{pb}
	Let  $\mathcal{A}=(A,A+\alpha_1\alpha_1^\T,\ldots,A+\alpha_m\alpha_m^\T)$ and $\mathcal{B}=(B,B+\beta_1\beta_1^\T,\ldots,B+\beta_m\beta_m^\T)$, where $A,B$ are $n\times n$ real  symmetric matrices, and $\alpha_i,\beta_i$ are nonzero vectors in $\mathbb{R}^n$ for all $i\in \{1,\ldots,m\}$. Suppose that $\mathcal{A}$ and $\mathcal{B}$ are entrywise similar.  Under what conditions can we guarantee that $\mathcal{A}$ and $\mathcal{B}$ are simultaneously similar via an orthogonal matrix $Q$ such that $Q^\T A Q=B$ and $Q^\T \alpha_i=\beta_i$ for $i=1,2,\ldots,m$?
\end{problem}
In the same paper, Wang and Zhao considered the special case that $\alpha_i=\beta_i=e_i$ for $i=1,\ldots,m$, where $e_1,\ldots,e_m$ are some nonzero 0-1 vectors such that any two of them are orthogonal (equivalently, the positions of the ones are pairwise disjoint).  For this case, they established the following theorem.

\begin{theorem}[\cite{wang2026AAM}]\label{genblock}
Let $J_{i,j}=e_ie_j^\T$ for $1\le i\le j\le m$. If $A+\sum s_{i,j}J_{i,j}$ and $B+\sum s_{i,j}J_{i,j}$  have the same characteristic polynomial for any $s_{i,j}$ where $1\le i\le j\le m$, then there exists an orthogonal matrix $Q$ such that $Q^\T A Q=B$ and $Q^\T e_i=e_i$ for $i=1,\ldots,m$.
\end{theorem}
In a recent manuscript \cite{xu2026}, under the additional assumption that $A$ and $B$ are adjacency matrices (or nonnegative matrices), Xu and Zhao simplify the condition of Theorem \ref{genblock} significantly. 
\begin{theorem}[\cite{xu2026}]\label{genblockdiag}
Let $J_{i,i}=e_ie_i^\T$ for $1\le i\le m$. Suppose that $A$ and $B$ are adjacency matrices of some graphs of the same order. If $A+\sum s_{i}J_{i,i}$ and $B+\sum s_{i}J_{i,i}$  have the same characteristic polynomial for any $s_{i}$ where $1\le i\le m$, then there exists an orthogonal matrix $Q$ such that $Q^\T A Q=B$ and $Q^\T e_i=e_i$ for $i=1,\ldots,m$.
\end{theorem}
Note that in Theorem \ref{genblockdiag}, the summation contains exactly $m$ items, whereas the corresponding summation in Theorem \ref{genblock} contains as many as $\binom{m+1}{2}$ items. The main aim of this paper is to generalize Theorems \ref{genblock} and \ref{genblockdiag} from the special case $\alpha_i=\beta_i=e_i$ to any vectors. The proof uses a technique introduced recently in \cite{wang2025EJC} to fill a small gap in the proof of a result of Farrugia \cite{farrugia2019EJC}.   
\begin{theorem}\label{main1}
	Let $A,B$ be two symmetric real matrices of order $n$ and $\alpha_1,\ldots,\alpha_m,\beta_1,\ldots,\beta_m$ be vectors in $\mathbb{R}^n$. Then the following two statements are equivalent:
	
\noindent\textup{(i)} There exists an orthogonal matrix such that $Q^\T AQ=B$ and $Q^\T \alpha_i=\beta_i$ for $i=1,\ldots,m$.

\noindent\textup{(ii)} Two matrices $A+\left(\sum_{i\in S}\alpha_i\right)\left(\sum_{i\in S}\alpha_i\right)^\T$ and $B+\left(\sum_{i\in S}\beta_i\right)\left(\sum_{i\in S}\beta_i\right)^\T$ are similar for any subset $S\subset \{1,\ldots,m\}$ with cardinality at most 2.
\end{theorem}
We note that the second statement in Theorem \ref{main1} means that (1) $\mathcal{A}=(A,A+\alpha_1\alpha_1^\T,\ldots,A+\alpha_m\alpha_m^\T)$ and $\mathcal{B}=(B,B+\beta_1\beta_1^\T,\ldots,B+\beta_m\beta_m^\T)$ are entrywise similar, and (2) 	$A+(\alpha_i+\alpha_j)(\alpha_i+\alpha_j)^\T\quad \text{and}\quad B+(\beta_i+\beta_j)(\beta_i+\beta_j)^\T $ are similar for any $i,j$ with $1\le i<j\le m$. Theorem \ref{main1} clearly improves upon Theorem \ref{genblock}. We make no assumption  on the vectors $\alpha_i,\beta_i$ in Theorem \ref{main1}. Moreover, the cospectrality condition in Theorem \ref{genblock} is greatly simplified and only symmetric matrices are involved in the second statement of Theorem \ref{main1}. 

We also obtain an improvement of Theorem \ref{genblockdiag} for nonnegative matrices and vectors, which we state as the following theorem.

\begin{theorem}	\label{main2}
	Let $A,B$ be two nonnegative symmetric real matrices of order $n$ and $\alpha_i,\beta_i$ be nonnegative vectors in $\mathbb{R}^n$ for $i=1,\ldots,m$. Then the following two statements are equivalent:
	
	\noindent\textup{(i)} There exists an orthogonal matrix such that $Q^\T AQ=B$ and $Q^\T \alpha_i=\beta_i$ for $i=1,\ldots,m$.
	
	\noindent\textup{(ii)} Two matrices $A+\sum_{i\in S}\alpha_i\alpha_i^\T$ and $B+\sum_{i\in S}\beta_i\beta_i^\T$ are similar for any subset $S\subset \{1,\ldots,m\}$ with cardinality at most 2.
	\end{theorem}	
\section{Proof of Theorems \ref{main1} and \ref{main2}}

Let $A$ be an $n\times n$ real symmetric matrix and $\lambda_1,\ldots,\lambda_s$ be all its distinct eigenvalues with multiplicities $r_1,\ldots, r_s$, respectively. Let $P_i$ be any $n\times r_i$ matrix whose columns consist of an orthonormal basis of $\mathcal{E}_{\lambda_i}(A)$, the eigenspace of $A$ corresponding to $\lambda_i$. Then $A$ has the spectral decomposition 
$$A=\lambda_1 P_1P_1^\T+\cdots+\lambda_s P_sP_s^\T.$$
We note that  $P_iP_i^\T$ is well defined although $P_i$ is not unique. Indeed, if $\tilde{P}_i$ consists of another orthogonal basis of $\mathcal{E}_{\lambda_i}(A)$, then we must have $\tilde{P}_i=P_iQ$ for some orthogonal matrix $Q$, which clearly implies $\tilde{P}_i\tilde{P}^\T_i= P_iP_i^\T$.

The following result can be proved using spectral decomposition and some standard skills. See, e.g.,  \cite{qiu2023} for details.
\begin{lemma}\label{samelen}
	Let $A$ and $B$ be two similar real symmetric matrices with spectral decomposition $A=\sum_{k=1}^s\lambda_k P_kP_k^\T$ and $B=\sum_{k=1}^s\lambda_k R_kR_k^\T$, respectively. For $\alpha,\beta\in \mathbb{R}^n$, if $A+\alpha\alpha^\T$ and  $B+\beta\beta^\T$ are similar, then $||P_k^\T \alpha||=||R_k^\T \beta||$ for $k=1,\ldots,s$.
\end{lemma}
A direct consequence of Lemma \ref{samelen} is the following
\begin{proposition}\label{real}
	Let $A$ and $B$ be two similar real symmetric matrices with spectral decomposition $A=\sum_{k=1}^s\lambda_k P_kP_k^\T$ and $B=\sum_{k=1}^s\lambda_k R_kR_k^\T$, respectively. For $\alpha,\beta\in \mathbb{R}^n$, if $(A+\alpha\alpha^\T,A+\tilde{\alpha}\tilde{\alpha}^\T, A+(\alpha+\tilde{\alpha})(\alpha+\tilde{\alpha})^\T)$ and  $(B+\beta\beta^\T,B+\tilde{\beta}\tilde{\beta}^\T, B+(\beta+\tilde{\beta})(\beta+\tilde{\beta})^\T)$ are entrywise similar, then  $\langle P_k^\T\alpha,P_k^\T\tilde{\alpha}\rangle =\langle R_k^\T\beta,R_k^\T\tilde{\beta}\rangle$  for $k=1,\ldots,s$.
\end{proposition}
\begin{proof}
	By Lemma \ref{samelen}, we have $||P_k^\T \alpha||=||R_k^\T \beta||$, $||P_k^\T \tilde{\alpha}||=||R_k^\T \tilde{\beta}||$, and $||P_k^\T (\alpha+\tilde{\alpha})||=||R_k^\T (\beta+\tilde{\beta})||$ for $k=1,\ldots,s$. Using the identity $||u+v||^2=||u||^2+||v||^2+2\langle u,v\rangle$, the last equality is equivalent to  
	$$||P_k^\T \alpha||^2+||P_k^\T\tilde{\alpha}||^2+2\langle P_k^\T \alpha, P_k^\T \tilde{\alpha}\rangle=||R_k^\T \beta||^2+||R_k^\T\tilde{\beta}||^2+2\langle R_k^\T \beta, R_k^\T \tilde{\beta}\rangle,$$
	and hence the proposition follows.
\end{proof}
With a little work,  we can obtain an additional version of Proposition \ref{real} for nonnegative case.
\begin{proposition}\label{rank2}
	Let $A$ and $B$ be two similar real symmetric nonnegative matrices with spectral decomposition $A=\sum_{k=1}^s\lambda_k P_kP_k^\T$ and $B=\sum_{k=1}^s\lambda_k R_kR_k^\T$, respectively. Let $\alpha,\tilde{\alpha},\beta,\tilde{\beta}$ be nonnegative vectors in $\mathbb{R}^n$. If $(A+\alpha\alpha^\T,A+\tilde{\alpha}\tilde{\alpha}^\T, A+\alpha\alpha^\T+\tilde{\alpha}\tilde{\alpha}^\T)$ and $(B+\beta\beta^\T, B+\tilde{\beta}\tilde{\beta}^\T,B+\beta\beta^\T+\tilde{\beta}\tilde{\beta}^\T)$ are entrywise similar, then $\langle P_k^\T\alpha,P_k^\T\tilde{\alpha}\rangle =\langle R_k^\T\beta,R_k^\T\tilde{\beta}\rangle$ for $k=1,\ldots,s$.
\end{proposition}
\begin{proof}
	Using the Weinstein-Aronszajn identity $|I_p-M_{p\times q}N_{q\times p}|=|I_q-N_{q\times p}M_{p\times q}|$, we obtain
	\begin{align}\label{hG}
		\left|xI-(A+\alpha\alpha^\T+\tilde{\alpha}\tilde{\alpha}^\T)\right|&=|xI-A|\cdot\begin{vmatrix}I_n-(xI-A)^{-1}(\alpha,\tilde{\alpha})\begin{pmatrix}\alpha^\T\\\tilde{\alpha}^\T\end{pmatrix}\end{vmatrix} \nonumber\\[5pt]
		&=|xI-A|\cdot\begin{vmatrix}I_2-\begin{pmatrix}\alpha^\T\\\tilde{\alpha}^\T\end{pmatrix}(xI-A)^{-1}(\alpha,\tilde{\alpha})\end{vmatrix}.\nonumber
	\end{align}
	By the spectral decomposition $A=\sum_{k=1}^s\lambda_k P_kP_k^\T$, we see that $$(xI-A)^{-1}=\sum_{k=1}^s\frac{1}{x-\lambda_k} P_kP_k^\T$$
	and hence
	\begin{align}
			\renewcommand{\arraystretch}{2.2}
		\begin{vmatrix}{\renewcommand{\arraystretch}{1}I_2-\begin{pmatrix}\alpha^\T\\\tilde{\alpha}^\T\end{pmatrix}(xI-A)^{-1}(\alpha,\tilde{\alpha})}\end{vmatrix}&=\begin{vmatrix}
			\displaystyle 1-\sum\limits_{k=1}^s\frac{||P_k^\T\alpha||^2}{x-\lambda_k} & \displaystyle -\sum\limits_{k=1}^s\frac{\langle P_k^\T \alpha, P_k^\T\tilde{\alpha}\rangle}{x-\lambda_k} \\
			\displaystyle -\sum\limits_{k=1}^s\frac{\langle P_k^\T \alpha, P_k^\T\tilde{\alpha}\rangle}{x-\lambda_k} & \displaystyle 1-\sum\limits_{k=1}^s\frac{||P_k^\T\tilde{\alpha}||^2}{x-\lambda_k}
		\end{vmatrix}. \nonumber
	\end{align}
	It follows that 
	\begin{equation}\label{aaa}
		\renewcommand{\arraystretch}{2.2}
		\left|xI-(A+\alpha\alpha^\T+\tilde{\alpha}\tilde{\alpha}^\T)\right|=|xI-A|\cdot\begin{vmatrix}
			\displaystyle 1-\sum\limits_{k=1}^s\frac{||P_k^\T\alpha||^2}{x-\lambda_k} & \displaystyle -\sum\limits_{k=1}^s\frac{\langle P_k^\T \alpha, P_k^\T\tilde{\alpha}\rangle}{x-\lambda_k} \\
			\displaystyle -\sum\limits_{k=1}^s\frac{\langle P_k^\T \alpha, P_k^\T\tilde{\alpha}\rangle}{x-\lambda_k} & \displaystyle 1-\sum\limits_{k=1}^s\frac{||P_k^\T\tilde{\alpha}||^2}{x-\lambda_k}
		\end{vmatrix}.
	\end{equation}
	Similarly, we have
	\begin{equation}\label{bbb}
		\renewcommand{\arraystretch}{2.2}
		\left|xI-(B+\beta\beta^\T+\tilde{\beta}\tilde{\beta}^\T)\right|=|xI-B|\cdot\begin{vmatrix}
			\displaystyle 1-\sum\limits_{k=1}^s\frac{||R_k^\T\beta||^2}{x-\lambda_k} & \displaystyle -\sum\limits_{k=1}^s\frac{\langle R_k^\T \beta, R_k^\T\tilde{\beta}\rangle}{x-\lambda_k} \\
			\displaystyle -\sum\limits_{k=1}^s\frac{\langle R_k^\T \beta, R_k^\T\tilde{\beta}\rangle}{x-\lambda_k} & \displaystyle 1-\sum\limits_{k=1}^s\frac{||R_k^\T\tilde{\beta}||^2}{x-\lambda_k}
		\end{vmatrix}.
	\end{equation}
	
	It is not difficult to see from the similarity assumption of Proposition \ref{rank2} together with Eqs. \eqref{aaa} and \eqref{bbb} that the two determinants of order 2 are equal. Moreover, by Lemma \ref{samelen}, we see that $||P_k^\T \alpha||=||R_k^\T \beta||$ and $||P_k^\T\tilde{\alpha}||=||R_k^\T\tilde{\beta}||$ for $k=1,\ldots,s$. It follows that
	$$
	\left(\displaystyle\sum\limits_{k=1}^s\frac{\langle P_k^\T \alpha,P_k^\T \tilde{\alpha}\rangle}{x-\lambda_k}\right)^2=\left(\displaystyle\sum\limits_{k=1}^s\frac{\langle R_k^\T \beta,R_k^\T \tilde{\beta}\rangle}{x-\lambda_k}\right)^2.
	$$
	This means that either $\langle P_k^\T\alpha, P_k^\T\tilde{\alpha}\rangle=\langle R_k^\T\beta, R_k^\T\tilde{\beta}\rangle$ for all $k$, or  $\langle P_k^\T\alpha, P_k^\T\tilde{\alpha}\rangle=-\langle R_k^\T\beta, R_k^\T\tilde{\beta}\rangle$ for all $k$.	Thus, to complete the proof of Proposition \ref{rank2}, it suffices to establish the following assertion.
	
	\noindent{\textbf{Claim}}: If $\langle P_k^\T\alpha, P_k^\T\tilde{\alpha}\rangle=-\langle R_k^\T\beta, R_k^\T\tilde{\beta}\rangle$ for all $k\in\{1,\ldots,s\}$, then 
	\[
	\langle P_k^\T\alpha, P_k^\T\tilde{\alpha}\rangle=\langle R_k^\T\beta, R_k^\T\tilde{\beta}\rangle=0 \quad \text{for all } k.
	\]
	
	Let $i$ be any integer in $\{0,1,\ldots,s-1\}$. From the spectral decomposition of $A$, we have 
	\begin{equation}\label{aAa}
		\alpha^\T A^{i}\tilde{\alpha}=\alpha^\T\left( \sum\limits_{k=1}^s
		\lambda_k^{i} P_kP_k^\T\right)\tilde{\alpha}=\sum\limits_{k=1}^s\lambda_k^i \langle P_k^\T \alpha,P_k^\T \tilde{\alpha}\rangle.
	\end{equation}
	Similarly, $\beta^\T B^{i}\tilde{\beta}=\sum\limits_{k=1}^s\lambda_k^i \langle R_k^\T \beta,R_k^\T \tilde{\beta}\rangle$. It follows from the condition of the Claim that $\alpha^\T A^i\tilde{\alpha}=-\beta^\T B^i \tilde{\beta}$. However, as the matrices $A,B$, and the vectors $\alpha,\tilde{\alpha},\beta,\tilde{\beta}$ are nonnegative, we must have $\alpha^\T A^i\tilde{\alpha}, \beta^\T B^i\tilde{\beta}\ge 0$.  This implies that all the numbers $\alpha^\T A^i\tilde{\alpha}$ and $\beta^\T B^i\tilde{\beta}$ are zero.  It follows from Eq.~\eqref{aAa} that 
	\begin{equation}\label{vd} 
		\sum\limits_{k=1}^s \lambda_k^i \langle P_k^\T\alpha, P_k^\T \tilde{\alpha}\rangle=0,\quad i=0,1,\ldots,s-1.
	\end{equation}
	As the Vandermonde matrix generated by $\lambda_1,\ldots,\lambda_s$ is invertible, we see from Eq.~\eqref{vd} that $\langle P_k^\T \alpha,P_k^\T \tilde{\alpha}\rangle=0$ for all $k\in \{1,\ldots,s\}$. As $\beta^\T B^i\tilde{\beta} =0$ for all $i$, the same argument indicates that  $\langle R_k^\T \beta,R_k^\T \tilde{\beta}\rangle=0$ for all $k\in \{1,\ldots,s\}$. This proves the Claim and hence completes the proof of Proposition \ref{rank2}.
\end{proof}

For a set of vectors $v_1,v_2,\ldots,v_m$ in $\mathbb{R}^n$, the \emph{Gram matrix} is the real symmetric matrix 
\begin{equation}
	\begin{pmatrix}
		v_1^\T v_1&v_1^\T v_2&\ldots&v_1^\T v_m\\
		v_2^\T v_1&v_2^\T v_2&\ldots&v_2^\T v_m\\
		\vdots&\vdots&\vdots&\vdots\\
		v_m^\T v_1&v_m^\T v_2&\ldots&v_m^\T v_m
	\end{pmatrix}
\end{equation}
We need the following basic result in Linear Algebra; see e.g. \cite[Theorem 7.3.11]{horn2013}.
\begin{lemma}\label{gm}
	Let $\{v_1,\ldots,v_m\}$ and $\{w_1,\ldots,w_m\}$ be two sets of vectors in $\mathbb{R}^n$. If they have the same Gram matrix, then there exists an orthogonal matrix $Q$ such that 
	$Q v_i=w_i$ for $i\in \{1,\ldots,m\}$.
\end{lemma}

\noindent\textbf{Proof of Theorem \ref{main1}}  Note that $A$ and $B$ are similar. Let $A=\sum_{k=1}^s\lambda_k P_kP_k^\T$ and $B=\sum_{k=1}^s\lambda_k R_kR_k^\T$ be the spectral decompositions. By Lemma \ref{samelen} and Proposition \ref{real}, we find that 
$\{P_k^\T\alpha_1,\ldots,P_k^\T\alpha_m\}$ and $\{R_k^\T \beta_1,\ldots,R_k^\T \beta_m\}$ have the same Gram matrix, for each $k\in \{1,\ldots,s\}$. It follows from Lemma \ref{gm} that, for each $k$, there exists an orthogonal matrix $Q_k$  such that $Q_k (P_k^\T \alpha_i)=R_k^\T \beta_i$ for $i\in\{1,\ldots,m\}$. Written in the matrix form, we have
\begin{equation}\label{mf}
	\begin{pmatrix}
		Q_1P_1^\T\\\vdots\\Q_sP_s^\T
	\end{pmatrix}(\alpha_1,\ldots,\alpha_m)=	\begin{pmatrix}
R_1^\T\\\vdots\\R_s^\T
	\end{pmatrix}(\beta_1,\ldots,\beta_m).
\end{equation}
Let
\begin{equation}\label{three}
Q=(P_1Q_1^\T,\ldots,P_sQ_s^\T)\begin{pmatrix}R_1^\T\\\vdots\\R_s^\T\end{pmatrix}.
\end{equation}
It is easy to see that both matrices on the right-hand side of Eq.~\eqref{three} are orthogonal matrices and hence $Q$ is also orthogonal. By Eq.~\eqref{mf}, we easily see that $Q^\T(\alpha_1,\ldots,\alpha_m)=(\beta_1,\ldots,\beta_m)$.
Finally, noting that
$$(P_1Q_1^\T,\ldots,P_sQ_s^\T)^\T A (P_1Q_1^\T,\ldots,P_sQ_s^\T)=\diag(\lambda_1 I_{r_1},\ldots,\lambda_sI_{r_s})$$
and 
$$(R_1,\ldots,R_s)^\T B (R_1,\ldots,R_s)=\diag(\lambda_1 I_{r_1},\ldots,\lambda_sI_{r_s}),$$
we obtain	\begin{eqnarray*}
	Q^\T A Q&=&(R_1,\ldots,R_s)(P_1Q_1^\T,\ldots,P_sQ_s^\T)^\T A(P_1Q_1^\T,\ldots, P_s Q_s^\T)\begin{pmatrix}
		R_1^\T\\ \vdots\\R_s^\T
	\end{pmatrix}\\
	&=&(R_1,\ldots,R_s)\diag(\lambda_1I_{r_1},\ldots, \lambda_{s}I_{r_s})\begin{pmatrix}
		R_1^\T\\ \vdots\\R_s^\T
	\end{pmatrix}\\
	&=&B.
\end{eqnarray*}
This completes the proof of Theorem \ref{main1}. 

\noindent\textbf{Proof of Theorem \ref{main2}}   Theorem \ref{main2} can be proved in exactly the same way as Theorem \ref{main1}. Indeed, by Lemma \ref{samelen} and Proposition \ref{rank2}, we find that $\{P_k^\T\alpha_1,\ldots,P_k^\T\alpha_m\}$ and $\{R_k^\T \beta_1,\ldots,R_k^\T \beta_m\}$ have the same Gram matrix for each $k\in \{1,\ldots,s\}$. Now the remaining proof of Theorem \ref{main1} clearly works for Theorem \ref{main2} and hence we are done.
	\section*{Declaration of competing interest}
There is no conflict of interest.
\section*{Acknowledgments}
This work is partially supported by the National Natural Science Foundation of China (Grant No. 12001006) and  Wuhu Science and Technology Project, China (Grant No.~2024kj015).


\begin{thebibliography}{99}
\bibitem{farrugia2019EJC}	A.~Farrugia, The overgraphs of 	generalized cospectral controllable graphs,  Electron. J. Combin. 26(1) (2019) \#P1.14.
\bibitem{guo2025AAM}S.~Guo, W.~Wang, W.~Wang, Primary decomposition theorem and generalized spectral characterization of graphs,  Adv. Appl. Math. 170 (2025) 102927.
	
\bibitem{horn2013}R.~A.~Horn, C.~R.~Johnson, Matrix Analysis (2nd ed.). Cambridge University Press, 2013.	
\bibitem{johnson1980JCTB} C.~R.~Johnson, M.~Newman, A note on cospectral graphs, J. Combin. Theory, Ser. B, 28 (1980) 96--103.	

\bibitem{qiu2023} L.~Qiu, Y.~Ji, L.~Mao, W.~Wang,
Generalized spectral characterizations of regular graphs based on graph-vectors, Linear Algebra  Appl. 663 (2023) 116--141.

\bibitem{wang2006EUJC}W.~Wang, C.-X.~Xu, A sufficient condition for a family of graphs being determined by their generalized spectra, European J. Combin. 27 (6) (2006) 826--840.
\bibitem{wang2017JCTB} W.~Wang, A simple arithmetic criterion for graphs being determined by their generalized spectra, J. Combin. Theory, Ser. B, 122 (2017) 438--451.
\bibitem{wang2025EJC} W.~Wang, W.~Wen, S.~Guo, A characterization of generalized cospectrality of rooted graphs with applications in graph reconstruction, Electron. J. Combin. 32(4) (2025)  \#P4.33
\bibitem{wang2023Eujc} W.~Wang, W.~Wang, F.~Zhu, An improved condition for a graph to be determined by its generalized spectrum, European J. Combin. 108 (2023) 103638.
\bibitem{wang2026AAM}W.~Wang, D.~Zhao, Graph isomorphism and multivariate
graph spectrum,  Adv. Appl. Math. 173 (2026) 102994.
\bibitem{xu2026} Y.~Xu, D.~Zhao, Generalized block diagonal Laplacian spectrum of graphs, arXiv:2511.23246.
\end{thebibliography}
\end{document}